\pdfoutput=1 
 \documentclass[12pt,reqno]{amsart}

\usepackage{amsmath}
\usepackage{amsthm}
\usepackage{amscd}
\usepackage{amssymb}
\usepackage{graphicx}
\usepackage{epsfig}
\theoremstyle{plain}
\newtheorem{theorem}{Theorem}[section]
\newtheorem{proposition}[theorem]{Proposition}
\newtheorem{lemma}[theorem]{Lemma}
\newtheorem{corollary}[theorem]{Corollary}

\newtheorem{definition}[theorem]{Definition}
\newtheorem{remark}[theorem]{Remark}

\newtheorem{fact}[theorem]{}

\def\int{\text{interior}}
\def\hyp {\hbox {\rm {H \kern -2.8ex I}\kern 1.25ex}}
\def\reals {\hbox {\rm {R \kern -2.8ex I}\kern 1.15ex}}
\def\integers {\hbox {\rm { Z \kern -2.8ex Z}\kern 1.15ex}}
\def\naturals {\hbox {\rm {N \kern -2.8ex I}\kern 1.20ex}}
\def\rationals {\hbox {\rm { Q \kern -2.2ex l}\kern 1.15ex}}
\def\hyp {\hbox {\rm {H \kern -2.7ex I}\kern 1.25ex}}
\def\bar{\overline}

\def\cal{\mathcal}

\def\strutdepth{\dp\strutbox}
\def \ss{\strut\vadjust{\kern-\strutdepth \sss}}
\def \sss{\vtop to \strutdepth{
\baselineskip\strutdepth\vss\llap{$\diamondsuit\;\;$}\null}}

\def\strutdepth{\dp\strutbox}
\def \sst{\strut\vadjust{\kern-\strutdepth \ssss}}
\def \ssss{\vtop to \strutdepth{
\baselineskip\strutdepth\vss\llap{$\spadesuit\;\;$}\null}}

\def\strutdepth{\dp\strutbox}
\def \ssh{\strut\vadjust{\kern-\strutdepth \sssh}}
\def \sssh{\vtop to \strutdepth{
\baselineskip\strutdepth\vss\llap{$\heartsuit\;\;$}\null}}

\def\strutdepth{\dp\strutbox}
\def \ssy{\strut\vadjust{\kern-\strutdepth \sssy}}
\def \sssy{\vtop to \strutdepth{
\baselineskip\strutdepth\vss\llap{{\large \bf Y}\,\,\,}\null}}

\def\strutdepth{\dp\strutbox}
\def \ssm{\strut\vadjust{\kern-\strutdepth \sssm}}
\def \sssm{\vtop to \strutdepth{
\baselineskip\strutdepth\vss\llap{{\large \bf M}\,\,\,}\null}}

\def\strutdepth{\dp\strutbox}
\def \ssti{\strut\vadjust{\kern-\strutdepth \sssti}}
\def \sssti{\vtop to \strutdepth{
\baselineskip\strutdepth\vss\llap{{\large \bf TIGHT}\,\,\,}\null}}

\def\strutdepth{\dp\strutbox}
\def \ssnn{\strut\vadjust{\kern-\strutdepth \sssnn}}
\def \sssnn{\vtop to \strutdepth{
\baselineskip\strutdepth\vss\llap{{\large \bf not nice}\,\,\,}\null}}

\def\strutdepth{\dp\strutbox}
\def \ssmm{\strut\vadjust{\kern-\strutdepth \sssmm}}
\def \sssmm{\vtop to \strutdepth{
\baselineskip\strutdepth\vss\llap{{ \bf small M-change}\,\,\,}\null}}

\def\strutdepth{\dp\strutbox}
\def \ssmM{\strut\vadjust{\kern-\strutdepth \sssmM}}
\def \sssmM{\vtop to \strutdepth{
\baselineskip\strutdepth\vss\llap{{ \bf rewritten by M}\,\,\,}\null}}

\def\strutdepth{\dp\strutbox}
\def \ssmq{\strut\vadjust{\kern-\strutdepth \sssmq}}
\def \sssmq{\vtop to \strutdepth{
\baselineskip\strutdepth\vss\llap{{ \bf M-question to Y}\,\,\,}\null}}

\begin{document}

\title{Are large distance Heegaard splittings generic ?}

\author{Martin Lustig and Yoav Moriah$^1$}

 \thanks{$^1$ This research is supported by  a grant from the High Council 
for Scientific and Technological Cooperation between France and 
Israel}

\date{\today}

\address{Math\'ematiques (LATP)\\
Universit\'e P. C\'ezanne - Aix-Marseille III \\
52 Ave. Escad. Normandie-Niemen, 13397 Marseille 20, France}
\email{martin.lustig@univ-cezanne.fr}

\address{Department of Mathematics \\
Technion \\
Haifa, 32000 Israel}
\email{ymoriah@tx.technion.ac.il}

\begin{abstract}    
In a previous paper \cite{LM2} we introduced a  notion of ``genericity" for countable sets of curves in the curve  complex of a surface $\Sigma$, based on the Lebesgue measure on the space of projective measured laminations in $\Sigma$.  With this definition we prove that for each fixed $g \geq 2$ the  set of  irreducible genus $g$ Heegaard splittings of high distance is generic, in the set of all irreducible Heegaard splittings.  Our definition of ``genericity" is different and more intrinsic then the one given via random walks.

\end{abstract}


\maketitle
 
\section{Introduction}
\label{introduction}

\vskip5pt

\subsection{The main result}\label{statement-main-result}

In this paper we use the definition of a generic subset of a countable set, introduced in Definition 4.1 of 
\cite{LM2} (see also subsection \ref{subsec_genericity} below), to  investigate whether for each fixed genus $g \geq 2$ the set of high distance Heegaard splittings 
is generic in  the set of all irreducible Heegaard splittings.  For this purpose, we consider irreducible Heegaard  diagrams composed of  two {\it complete decomposing systems}   $\cal D, \cal E$ on a closed surface $\Sigma$,  i.e. two curve systems which both decompose $\Sigma$ into pair-of-pants. We define the set $SNOW(\Sigma)$ of  such pairs $(\cal D, \cal E)$, by  requiring  the pairs to have the additional property that they do not have {\it waves}  (see Definition \ref{smallwave}) with respect to each other.  A pair of such curve systems,  which  determine a Heegaard  splitting that is of distance greater or equal to some  $n \in \mathbb{N}$, will be said to belong to  $SNOW_n(\Sigma)  \subset SNOW(\Sigma)$.  For precise definitions see Sections \ref{prelims} and \ref{tt} below. Our main result is:

\vskip5pt

 \noindent {\bf Theorem \ref{large_distance_snows_are_generic}.} {\it   Let $\Sigma$ be a closed surface of genus $g \geq 2$.  Then for any integer $n \geq 1$ the set $SNOW_{n - 1}(\Sigma)$ is generic in the set  $SNOW(\Sigma)$. }

\vskip10pt

The fact that a Heegaard splitting is high distance has important consequences for the geometry of the  $3$-manifold determined by it. Hence the question, raised in the title of this paper, is natural and relevant. Indeed, since the distance is 
measured in the {\it curve complex}\, $\cal C(\Sigma)$, which is not locally finite, it is a priory not clear, despite of the fact 
that $\cal C(\Sigma)$ has infinite diameter, whether one should expect random points in $\cal C(\Sigma)$ to be close or far 
away from each other. The question becomes even less evident if asked for  subsets of $\cal C(\Sigma)$ with infinite diameter, 
such as  the handlebody sets $\cal{D}(V)$ and  $\cal{D}(W)$, which are used to determine the distance of a Heegaard splitting 
$M = V \cup_\Sigma W$, see section \ref{prelims}.

There are several ways to attack this question, and this paper presents one which is described in
detail in subsection \ref{MandO}.
J. Maher \cite{Mah} has provided another such approach,
based on random walks in the mapping class group and the issuing notion of a ``random 3-manifold'', which was introduced and investigated by N. Dunfield and W. Thurston in \cite{DT}.

\vskip10pt

 \subsection{The notion of ``genericity'' used in this paper}\label{subsec_genericity}
 
It is well known that making mathematical sense of the term ``generic'' is  particularly problematic when considering subsets of countable sets $Y$.  In the ``classic'' setting a set $A \subset Y$ would be called 
{\it generic}  if and only if the complement $Y \smallsetminus A$ is finite. This definition is very restrictive and, in practice, there are many cases where intuitively one would like to call a set generic, yet it does not satisfy this criterium. A common method to deal with this situation is  the following:  One considers a family of finite subsets $Y_n \subset Y$ that exhaust $Y$ (typically given by enumerating 
$Y = \{ y_1, y_2, \ldots \}$, and by defining $Y_n = \{ y_1, y_2, \ldots, y_n \}$). One then computes the the ratio $\rho_n = \frac{\#(Y_n \cap A)}{\#Y_n}$, and defines $A$ to be generic if $\rho_n$ tends to 1.
Of course, the obtained limit ratio depends heavily on the particular chosen enumeration of $Y$.

Another difficulty with statements about  ``genericity''  is that of {\it double counting}.   A typical example here is to count group presentations instead of groups, so that the same group may well be counted several times, or even infinitely often. The problem arises  when in the given  mathematical context it is impossible to count the objects in question directly,  so that double counting becomes unavoidable.

 \vskip10pt

In this paper we do  indeed consider a countable set $Y$ as base set, but we use the fact that our set 
$Y$ is  given as subset of a space $X$  that is already equipped with a natural measure, or rather, 
measure class.  

 This enables us to use the following definition, which was  introduced in a previous paper \cite{LM2}.
It is motivated by the observation that a ``generic'' point with rational coordinates in the unit square 
$[0, 1] \times [0, 1]$ should rightfully be expected to lie in the interior rather than on the boundary of 
$[0, 1] \times [0, 1]$.

\vskip5pt

\begin{definition}\label{generic} \rm
Let $X$ be a topological space,  provided with a Borel measure  $\mu$. Let $Y \subset X$   be a (possibly countable) subset, which is a disjoint  union   $Y =  A \overset {\cdot}{\cup}  B$.   The set $A$ is called {\it generic in $Y$} (or simply {\it  generic},  if  $Y = X$) if the closure  $\bar A$ of $A$ has measure  $\mu(\bar A) > 0$, and  the closure    $\bar B$ of $B$ has   measure $\mu(\bar B) = 0$. (Here ``closure'' always means ``closure in $X$''.)

\end{definition}

Notice that in this definition the sets $\bar  A$  and $\bar B$ may well not be disjoint, although $A$
and $B$ are  assumed to be disjoint.  Note also, that this definition of genericity  extends to sets $Y$ that are not embedded but are just mapped to $X$, by a properly chosen ``natural'' map.  Furthermore, note that this definition does not depend on the actual measure $\mu$ but rather only on its measure class.

 \vskip5pt

The following is an immediate  consequence of the above definition:

\vskip5pt

\begin{lemma}  [Lemma 4.3 (1) of \cite{LM2}] \label{equivalentgeneric}
Given sets $X, Y$ and $A$ as in Definition \ref{generic}. Then $A$  is  generic in $Y$ if and only if  the
closure $\bar A$ contains a set $Z$ which is open in  $\bar Y$ and  of full measure
$\mu(Z) = \mu(\bar Y) > 0$, and which is disjoint from $Y  \smallsetminus A$.

\qed

\end{lemma}

We will specify precisely in the next subsection the sets $X, Y, A$ and the measure $\mu$ that are used in this paper.

 \vskip15pt
 
\subsection{Measures and objects}\label{MandO} The Borel measure space $X$, as in the last subsection, that is used in this paper is the space  $\cal{PML}(\Sigma) \times \cal{PLM}(\Sigma)$,  where 
$\cal{PML}(\Sigma)$ is the space of projective measured laminations on a ``model surface'' 
 $\Sigma$ of genus $g \geq 2$.  As shown by Thurston, the space $\cal{PML}(\Sigma)$ is a 
 $(6g - 7)$-dimensional sphere, equipped with a canonical p.l.-structure, and thus with a canonical Lebesgue measure class.

\vskip5pt

The countable subset $Y \subset X$, required by this definition of the term ``generic'' in the statement of Theorem \ref{large_distance_snows_are_generic},  is a set of pairs $(\cal D, \cal E) \in \cal{PML}(\Sigma)^2$, called $SNOW(\Sigma)$, where $\cal D$ and $\cal E$ are curve systems  on $\Sigma$  that satisfy a certain combinatorial condition, see Definition \ref{snow}. However, in order to focus on the objects of our primary interest, namely Heegaard splittings, we have to be more careful,
in order to avoid a potential double-counting problem as evoked in the previous subsection:

\vskip5pt
A {\it marked Heegaard surface of genus $g$} is a surface $S_g$ of genus $g$ embedded in a 3-manifold $M$, with handlebody complementary components $V$ and $W$, which is equipped with 
a {\it marking homeomorphism} $\theta: \Sigma \to S_g$. Two such marked Heegaard surfaces, with markings $\theta$ and $\theta'$, are {\it equal} if there exists a homeomorphism $h: M \to M$ which is isotopic to the identity, such that $\theta' = h \circ \theta$. Let $\cal S^g$ denote the set of marked Heegaard surfaces of genus 
$g \geq 2$, and let $\cal {S}^g_n$ be the subset of those Heegaard surfaces of distance $\geq n$. Here the {\it distance} of a marked Heegaard surface $S_g$ is measured in the curve complex of the model surface $\Sigma$: More precisely, one measures the distance between the two handlebody sets 
$\cal D(V)$ and $\cal D(W)$ that are given by the $\theta^{-1}$-images of all meridian curves in $S_g$, for each of the two complementary handlebodies $V$ and $W$ respectively.

\vskip5pt

Now choose a map $\sigma_{\cal S}: \cal S^g \to SNOW(\Sigma), (S_g, \theta) \mapsto (\cal D, \cal E)$, where $\theta(\cal D)$ and $\theta(\cal E)$ are systems of meridian curves on  $V$ and $W$ respectively. There are infinitely many such maps, and none seems to be prefered in any obvious way.  In Section  \ref{generic-H-splittings}  we prove for any such map $\sigma_{\cal S}$:

\begin{corollary} 
\label{large_distance_HSs_are_generic_intro}
For any integers $g \geq 2$ and $n \geq 0$ the set $\sigma_{\cal S}(\cal{S}^g_n)$ is generic in the set 
$\sigma_{\cal S}(\cal{S}^g)$.
\end{corollary}

\vskip10pt

{\bf Acknowledgments.} We would like to thank the High Council  for Scientific and Technological Cooperation between France and  Israel for supporting this project, the Department of Mathematics 
at the Technion Haifa Israel and LATP  at the Universit\'e P. C\'ezanne - Aix-Marseille III in Marseilles France for their hospitality. 

\vskip20pt

  \section{Preliminaries} \label{prelims}
 
 \vskip10pt
 
\noindent  
Most of the  definitions we recall here are standard. We use the terminology of our previous papers 
\cite{LM1} and \cite{LM2}, where also some more details can be found.  Definitions  \ref{snow} (b) and \ref{SNOWn} are new.

\vskip10pt

Throughout this paper $\Sigma$ will denote a closed orientable surface of genus $g \geq 2$.  We will frequently consider essential simple closed curves $D_i$ (or $E_j$) on $\Sigma$, and we will not distinguish notationally between curves and isotopy classes of them. 

\vskip5pt

A curve $D$  is called {\it tight} with respect to a system  $\mathcal E = \{E_{1}, \ldots, E_{r}\}$ of pairwise disjoint essential simple closed
curves  $E_i$ in $\Sigma$  if the number of  intersection points with 
$\mathcal E$ can not be strictly decreased  by an isotopy of  $D$. The same terminology is used 
for arcs $\alpha$ which have their  endpoints on $\mathcal E$, where the endpoints cannot leave 
$\mathcal E$ throughout  the isotopy.

\vskip10pt

\begin{definition}\label{smallwave}\rm  
Let $P \subset \Sigma$ be a pair-of-pants,  i.e. a sphere with three open disks removed. 

\vskip7pt

\begin{enumerate}
\item[(a)]  A simple arc in $P$ which has its two endpoints on different  components of $\partial P$ 
will  be called a {\it seam}.

\item[(b)]  A simple arc in $P$ which has  both endpoints on the same  component of $\partial P$, 
and is not  $\partial$-parallel, will be called a {\it wave}.

\item[(c)]   An essential simple closed  curve $D \subset \Sigma$ has a  wave   (or a seam) 
with respect to a  system of curves $\mathcal{E}  \subset \Sigma$ if $D$ is tight with respect to 
$\mathcal E$ and if $D$ contains a subarc  that is a wave (or a seam) in a  complementary component 
$P_{i}$ of  $\mathcal{E}$ in $\Sigma$ which is a pair-of-pants.  

\end{enumerate}

\end{definition}

 \vskip5pt
 
\begin{definition} \label{curvecomplex}\rm 
For any closed surface $\Sigma$ as above we define the curve complex  ${\cal C}(\Sigma)$ as 
follows: A closed $m$-simplex is the isotopy class of a set $\{D_0, \ldots, D_m\}$ of mutually disjoint 
non-parallel essential simple  closed curves on $\Sigma$. On the $1$-skeleton  ${\cal C}^1(\Sigma)$ 
of ${\cal C}(\Sigma)$ there  is a well defined metric $d_\cal C$
given by assigning length $1$ to every edge (= a 1-simplex).

\end{definition} 

The complex ${\cal C}(\Sigma)$  is of finite dimension $3g - 4$ and is not locally finite. It is known (see \cite{MM})  to be hyperbolic in the sense introduced by Gromov, and it admits a natural action of the mapping class group of $\Sigma$.

 \vskip5pt

\begin{definition}\label{snow} \rm 

(a) The set  $\cal D = \{D_0, \ldots, D_{3g-4} \}$ of (isotopy classes of) curves in $\Sigma$ defined 
by a maximal dimensional simplex  in $\cal C(\Sigma)$  is called a {\em complete decomposing system} of $\Sigma$.  All of its complementary components are pairs-of-pants. The set of complete decomposing systems of $\Sigma$ (up to isotopy) is denoted by $\cal{CDS}(\Sigma)$. There is a canonical identification between $\cal{CDS}(\Sigma)$ and the set of maximal dimensional simplices 
of $\cal C(\Sigma)$.

\vskip5pt

\noindent
(b)  Any two complete decomposition systems ${\cal D},  {\cal D}' \in \cal{CDS}(\Sigma)$  satisfy the   
{\it symmetric no wave condition} (or ${SNOW}$), if each curve in ${\cal D}$ has no waves with respect 
to any of the curves of ${\cal D}'$, and vice versa. Define $SNOW(\Sigma)$ to be the set of all pairs of complete decomposing systems which are $SNOW$.  

\end{definition}

\vskip5pt

Any  complete decomposing system $\cal D =  \{D_0, \ldots, D_{3g-4}\} \subset {\cal C}(\Sigma)$ determines a handlebody $V_{\cal D}$ with boundary $\partial V_\cal D = \Sigma$, which is defined by  
by declaring  that all the curves $D_i$ bound disks in  $V_{\cal D}$.  Conversely, given a handlebody 
$V$ with an identification $\partial V = \Sigma$,  then the set ${\cal D}(V)$ of curves in $\cal C^0(\Sigma)$ which bound disks in $V$ (called {\it meridian curves} on $\partial V$) is called the  {\it handlebody set} determined by $V$.  Such a set ${\cal D}(V)$ is the vertex set of a 
subcomplex of ${\cal C}(\Sigma)$ which is connected (see \cite{Mc}).

\vskip5pt

The following is well known and easy to prove:

\begin{remark} \label{otherdisks1} \rm Let ${\mathcal D} \in \mathcal{CDS}(\Sigma)$ be a complete decomposing system, with associated handlebody $V_\cal D$ as above. Then any meridian curve $D \in \cal D(V_\cal D)$   is either contained in $\cal D$,  or else $D$  has a wave with respect to $\mathcal D\,$.

\end{remark}

Any  two  maximal dimensional simplexes  $\sigma,\tau \subset\cal C(\Sigma)$ determine a Heegaard splitting  $M = V \cup_{\Sigma} W$ for some $3$-manifold $M$. Here $V$ and $W$ are the handlebodies determined by  the two complete decomposing systems that are given by  $\sigma$ and $\tau$ 
respectively. 

Given a Heegaard splitting $M = V \cup_{\Sigma} W$ of some $3$-manifold $M$, Hempel (see \cite{He}) defines  the {\em distance of the Heegaard splitting} as follows:

$$ d(V,W) = {\min}\{d_\cal C(D, E) \mid D \in {\cal D}(V), \, E \in {\cal D}(W)\}$$

Note that a handlebody set has infinite diameter in $\cal C(\Sigma)$. Hence, for randomly chosen $D \in \cal D(V), E \in \cal D(W)$, the distance $d_\cal C(D, E)$ may be arbitrary large. 
However, combinatorial conditions for choosing curves $D, E$ which realize $d(V,W)$ are given in \cite{LM1}.

\vskip5pt

A crucial ingredient  in the arguments of this paper is Lemma 1.3 of Hempel ~\cite{He} which can be restated as follows:

\begin{lemma}\label{hempellemma} 
Any irreducible Heegaard splitting  $M = V \cup_{\Sigma} W$ admits a pair of   complete decomposing systems  $D \in {\cal D}(V), \, E \in {\cal D}(W)$ which satisfy the condition $SNOW$.

\end{lemma}

\begin{definition}\label{SNOWn}\rm
A pair $({\cal D}, {\cal E}) \in{\cal CDS}(\Sigma)$ is said to have property $SNOW_n$, if $({\cal D}, {\cal E}) \in SNOW(\Sigma)$ and if in addition the handlebodies $V$ and $W$, determined by $\cal D$ and $\cal E$ respectively, satisfy $d(V,W) \geq n$. The set of all such pairs $({\cal D}, {\cal E})$ is denoted by $SNOW_n(\Sigma)$.

\end{definition}

\vskip15pt

 \section{Train tracks and Laminations }\label{tt}
 
 \vskip5pt

 \subsection{Basics} \label{basics-on-tts}\hfill
 
 \vskip10pt

In this subsection we recall some basic definitions and notations. For a more detailed exposition see \cite{LM1} and \cite{LM2}:
 
 \vskip10pt

 \begin{fact}
 \label{train-track-basics} \rm
(a) A {\it train track} $\tau$ in $\Sigma$ is a closed subsurface with a  {\it singular $I$-fiberation}: The interior of  $\tau$ is fibered by open arcs,   and   the  fibration extends to a fiberation of the closed surface  $\tau$  by  properly  embedded closed arcs (the {\it $I$-fibers}), except for finitely many  {\it singular points}  (also called  {\it cusp points}) on $\partial \tau$,  where precisely two  fibers meet. We call these  fibers {\it singular fibers}. We admit  the case that a fiber is {\it  doubly singular}, i.e. both of  its  endpoints are singular points. 

\vskip5pt

\noindent
(b) Two singular fibers are {\it adjacent} if they share a singular  point  as a common endpoint. A maximal connected union of singular  or  doubly singular $I$-fibers is called an {\it exceptional fiber}.  It is either 
homeomorphic  to a closed  interval, or to a simple closed curve on  $\Sigma$.  In  the latter case it will be called a {\it  cyclic  exceptional fiber}. We explicitely admit this second case, although  in the classical  train track  literature this case is sometimes  suppressed.

\vskip5pt

\noindent
(c)  Following Thurston, one usually pictures a singular point $P \in \partial \tau$  in such a way 
that any two arcs from $\partial \tau$, which intersect in $P$, converge towards  $P$ from the ``same direction", thus giving  rise to a {\it cusp point} on the boundary of the corresponding complementary component  of  $\tau$ in  $\Sigma$.

\vskip5pt

\noindent
(d) We define the {\it type} of a complementary component $\Delta$ of   $\tau$ in  $\Sigma$ as given by 
the genus of $\Delta$,  the number of boundary components of $\Delta$, and the number of cusp points on each of its boundary components.  If $\Delta$ is simply connected, we speak of  an $n$-gon if there are  precisely $n$ cusp points on $\partial  \Delta$.  For example, if  $\partial \Delta$ contains precisely three  cusp points, we say that  $\Delta$ is a {\it triangle}. An arc of   $\partial \Delta$ which joins two adjacent cusp points is called a    {\it side} of $\Delta$. 

If $\Delta$ is simply connected, one usually requires $\Delta$ to have at least 3 cusp points on $\partial \Delta$.

\vskip10pt

\noindent
(e) A train track $\tau$ in $\Sigma$ is called {\it filling}, if all  complementary components of $\tau$ in 
$\Sigma $ are simply  connected,  and if each of them has at least 3 cusp points on its boundary. 
The train track $\tau$ is called {\it maximal}, if every  complementary component is a triangle.

\end{fact}

\vskip5pt  

\begin{fact} \label{laminations-basics} \rm
(a)  A lamination $\cal L$ in $\Sigma$ is a  non-empty  closed subset $\Sigma$ which is the union 
of (possibly non-closed)  simple curves,  called {\it leaves}, that are geodesic with respect to some hyperbolic structure on $\Sigma$. A lamination $\cal L$ provided with a transverse measure  $\mu$ 
 is called a measured lamination, and its projective class is denoted by $[\cal L, \mu]$. The space 
 $\cal{PML}(\Sigma)$ of all projective measured laminations $[\cal L, \mu]$ on $\Sigma$, where 
 $\cal L$ is the support of $\mu$,  is known to be homomorphic to a $(6g-7)$-dimensional sphere, 
 often refereed to as the ``Thurston boundary of Teichm\"uller space''.

\vskip5pt  

\noindent
(b) It is well known that $\cal{PML}(\Sigma)$ contains a dense subset of {\em uniquely ergodic} laminations:  In this case $[\cal L, \mu]$ is uniquely determined by $\cal L$.  Hence, in this case 
we can write  $\cal L \in \cal{PML}(\Sigma)$.  In particular,  every essential simple closed curve 
$D \subset \Sigma$  by itself  is a uniquely ergodic lamination, so that one has a canonical embedding 
$\cal C^0(\Sigma) \subset \cal{PML}(\Sigma)$.

A complete decomposing system $\cal D$ on $\Sigma$ is a lamination that is not uniquely ergodic, so that it defines a finite dimensional simplex in $\cal{PML}(\Sigma)$.  By associating to $\cal D$ the barycenter of this simplex (i.e. the curves of $\cal D$ all carry the same weight) we obtain a well defined embedding $\cal{CDS}(\Sigma) \subset \cal{PML}(\Sigma)$.

\vskip5pt  

\noindent
(c) A lamination $\cal{L}$ in a surface $\Sigma$ is called {\it maximal} if all  of  its complementary regions are triangles.  A lamination $\cal{L} \subset \Sigma$ is called {\it minimal} if every leaf is dense in  
$\cal{L}$.  (This terminology refers to the partial order on the set of laminations given by the inclusion.)
A lamination which satisfies both of these conditions is called {\em minimal-maximal}. 
Such laminations are {\em totally arational}, i.e. they do not contain any closed curve as leaf.

Again, the set of uniquely ergodic minimal-maximal laminations is  well known to be dense in 
$\cal{PML}(\Sigma)$.
\end{fact}

\vskip5pt

\begin{fact} \label{carried-parallel} \rm
(a) An arc, a closed curve or a lamination in $\Sigma$ {\it is carried}  by a train track $\tau \subset \Sigma$ if it is contained in $\tau$ and  throughout transverse to the  $I$-fibers of $\tau$.   A curve {\it can be carried} by $\tau$ if after a suitable isotopy it is carried by $\tau$. The set of projective measured laminations $[\cal L, \mu]$ which have as support a lamination $\cal L$ carried by $\tau$ is denoted by 
$\cal{PML}(\tau)$.

\vskip5pt  

\noindent
(b) Two simple arcs carried by $\tau$ are {\it parallel} if  they intersect  the same $I$-fibers,  and these intersections occur on the two arcs  in precisely the same order. An arc, a closed curve or a lamination on  $\Sigma$ which is  carried by $\tau$ is   said to  {\it cover} $\tau$ if  it  meets every $I$-fiber of $\tau$.

\vskip5pt  

\noindent (c) A train track $\tau'$ is {\it carried by a train track} $\tau$ if every $I$-fiber of $\tau'$ is a subarc of an $I$-fiber of  $\tau$. Note that every curve (or arc or lamination) carried by $\tau'$ is also carried by  $\tau$.

\vskip5pt

\noindent
(d) Given a train track $\tau$ and a lamination $\cal{L}$ carried by it, the notion of a ``derived''  train 
$\tau_1$ with respect to $\cal{L}$ is formally defined in subsection 2.3 of \cite{LM2}.  In slightly informal terms,  $\tau_1$ is the train track obtained from $\tau$ by splitting $\tau$ at every cusp point along a 
{\it unzipping path}, which is a path disjoint from $\cal{L}$ that covers $\tau$. If $\cal{L}$ is maximal-minimal, such a  train track always exists and is uniquely determined by $\tau$ and $\cal{L}$. The train  track $\tau_1$  obtained by splitting along  shortest possible such paths is  said to be  {\it  derived from 
$\tau$ with respect to  $\mathcal L$}, or simply {\it  derived from  $\tau$}.

Notice that $\cal L$ is always carried by $\tau_1$, and that $\tau_1$ is carried by $\tau$. More generally, every sequence of unzipping paths leads to a new train track that is carried by the original one.

\vskip5pt 

\noindent
(e) If $\tau'$ is derived from $\tau$, then every curve (or lamination) that is carried by $\tau'$ covers 
$\tau$, by Lemma 2.9 of \cite{LM1}.  The proof of this lemma shows that the same is true for any arc which runs parallel on $\tau'$ to at least one entire side of one of the connected components complementary to $\tau'$, or parallel to one of the unzipping paths used to derive $\tau'$ from $\tau$. 

\end{fact}

\vskip5pt

\begin{definition} \label{ngregarious}\rm
Let $\tau \subset \Sigma$ be a train track, and let $\mathcal L$ be  a  lamination carried by $\tau$.
We say that $\mathcal L$ is {\it  $n$-gregarious} with respect to    $\tau$ if $\tau$ can be derived $n$ times with respect to  $\mathcal L$,  i.e. there exists an {\it $n$-tower of derived train tracks} 
$$\tau = \tau_{0} \supset \tau_{1} \supset \ldots \supset \tau_{n}$$ 
such that   $\mathcal L$ is carried by $\tau_{n}$.

\end{definition}

\vskip10pt

We conclude  this subsection by stating  a result from an earlier paper which is  crucially used later on:

\vskip10pt

\begin{proposition}[Proposition 2.12 in \cite{LM1}] \label{towerimpliesdistance} 
Let $\tau_{0} \supset \tau_{1} \supset \ldots \supset\tau_{n}$  be  an  $n$-tower of derived  train tracks 
in $\Sigma$. Assume that  $\tau_{0}$   (and hence of any of the  $\tau_{i}$) is maximal.  Let  $D$ be a simple closed curve  carried  by $\tau_n$.  Then any  essential simple closed curve $D'$ which  satisfies
$$d_{\mathcal{C}}(D, D')\leq n$$
is carried by  $\tau_{0}$.

\end{proposition}

\vskip15pt

\subsection{Tight train tracks}\label {Tighttraintracks}\hfill

\vskip10pt

Part (a) of the following definition has been introduced in ~\cite{LM2}:

\vskip10pt

\begin{definition}\label{tight}\rm 
(a)  A maximal train track $\tau \subset \Sigma$ is called {\it tight}  with respect to  some complete  decomposing system $\cal{E}$ on $\Sigma$ if the following conditions  are satisfied: 

\vskip5pt

\begin{enumerate}

\item[(1)] For any curve $E_k \in \cal{E}$  the intersection  $E_k \cap \tau$ is  a disjoint union of (possibly exceptional) $I$-fibers of $\tau$. 

\vskip5pt

\item[(2)] For every connected component $\Delta_j$ complementary to $\tau$ the intersection segments with any $E_k \in \cal{E}$  are arcs with endpoints  on two distinct sides of  $\Delta_j$.

\vskip5pt

\item[(3)] Each of the three cusps of any complementary component $\Delta_j$ is contained in some of the $E_k$. 
\end{enumerate}

\vskip5pt
\noindent
(b) A train track $\tau \subset \Sigma$ is called {\em tight} if it is tight with respect to some complete decomposing system $\cal E \in \cal{CDS}(\Sigma)$.

\end{definition}

\vskip5pt

The condition (3) of Definition \ref{tight} (a) is equivalent to stating that every singular $I$-fiber of $\tau$ lies on some of the curves $E_i \in \cal{E}$ .  We also would like to point out that the notion of a tight train track is a fairly general one: for example, a train track which is tight with respect to $ \cal{E}$ may well carry a wave with respect to $ \cal{E}$.

\vskip10pt

\begin{proposition}[Proposition 2.11 \cite{LM2}] \label{carriedcurves}
Let $\mathcal E$ be a complete decomposing system of $\Sigma$, and let $\tau$ be a maximal train track that is tight with respect to $\mathcal E$. Let $c$ be a simple closed curve  (or a finite collection of such) on $\Sigma$ that is tight with respect to $\mathcal E$ and contains a subarc $\beta$ which covers $\tau$.  Then $c$ can be carried by $\tau$.

\qed

\end{proposition}

\vskip5pt

Complete decomposing systems $\cal D$ carried by train tracks have been investigated in  \cite{LM1}.  
It has been shown there that, for any maximal train track $\tau$, every pair-of-pants $P$ which is complementary to such a system ${\cal D}$ is either of {``eye glass shape''} or {``$\Theta$-graph shape''}.  That is: $P$ always contains precisely two of the  triangles that are  complementary components of $\tau$, say $\Delta_1$ and  $\Delta_2$. More precisely, $P$ is the regular neighborhood of the union of $\Delta_1$ and $\Delta_2$,  and of three arcs $\gamma_1, \gamma_2$  $\gamma_3$ that are carried by 
$\tau$ and have the cusps of  $\Delta_1$ and $\Delta_2$ as endpoints.  We say  that $P$  {\em has 
$\Theta$-graph shape} if  every  $\gamma_k$, for  $k = 1, 2, 3$, connects a cusp of  $\Delta_1$ to a cusp of $\Delta_2$.  On the other hand,  $P$  has  {\it eye glass shape} if only one of the  $\gamma_k$ connects a cusp of $\Delta_1$ to a cusp of $\Delta_2$, while each of the other two  $\gamma_h$ join two cusps from the same $\Delta_i$. 

\vskip10pt

\begin{remark} \label{theta-waves} \rm
If $P$ is $\Theta$-graph shaped, then any wave $\beta$ in $P$  can be  carried by $\tau$. Furthermore 
$\beta$ has to run parallel  on $\tau$ to at least one entire side of one  of  the two  connected  components complementary to $\tau$ which are contained in $P$. This is an easy consequence of the above definitions (compare also Lemma 4.4 of \cite{LM1}).

\end{remark}

\vskip10pt

The following lemma has been stated and proved as Lemma 4.6 in \cite{LM1}, for the special case where 
$\tau$ is a complete fat train track. We follow the proof given there, except for the quote of Lemma 3.9 of 
\cite{LM1}, which has been generalized and corrected in Lemma 2.11 of \cite{LM2}, 
see Proposition \ref{carriedcurves} above.

\vskip7pt

\begin{lemma} \label{carriedwaves}
Let $\tau \subset \Sigma$ be a  maximal tight  train track, and let $\tau'  \subset \tau$ be a train
track derived from $\tau$.  Let $\mathcal{D}$ be a complete decomposing system in $\Sigma$
which is carried by $\tau'$,  with the property that every  pair-of-pants complementary to
$\mathcal  D$ has $\Theta$-graph shape. Let  $D \subset \Sigma$  be an essential simple
closed curve which is tight with  respect to $\mathcal D$, and assume  that some arc $\beta$ from
the set of  arcs $D - \mathcal D$  is a wave with respect to  $\mathcal D$. Then $D$  can be carried by
$\tau$.

\end{lemma}

\vskip10pt

\begin{proof}
By assumption $\tau$ is tight with respect to some system $\cal E \in \cal{CDS}(\Sigma)$. We can then isotope the curve $D$ so that it is tight with respect to $\cal E$, while keeping it tight with respect to $\cal D$. This can be done, for example, by making every curve geodesic with respect to some hyperbolic structure on $\Sigma$.

Let $P$ be the pair-of-pants complementary to $\mathcal D$ that   contains the wave $\beta$. By assumption  $P$ has $\Theta$-graph  shape,  so that by Remark \ref{theta-waves}  the wave 
$\beta$ is carried by $\tau'$, and $\beta$ has to run parallel  on $\tau'$ to at least one entire side of one  of  the two  connected  components complementary to $\tau'$ which are contained in $P$. Thus, by 
paragraph  \ref{carried-parallel} (e) of subsection \ref{basics-on-tts}, $\beta$ covers $\tau$. Hence we can apply Proposition  \ref{carriedcurves} to conclude that $D$ is carried by $\tau$.

\end{proof}

\vskip10pt

\begin{corollary} \label{pre-distance}
Let $\tau$ be a maximal tight train track. For some integer $n \geq 1$, consider a complete decomposing system  $\cal D \in \cal{CDS}(\tau)$ which is $n$-gregarious with respect to $\tau$, with the property that every complementary components has $\Theta$-graph shape. Then every meridian curve $D \in \cal D(V_\cal D)$ for the handlebody $V_\cal D$ determined by $\cal D$ is $(n-1)$-gregarious with respect to 
$\tau$.  

More specifically, if $\cal D$ is carried by $\tau_n$, for some $n$-tower of derived train tracks 
$\tau = \tau_0 \supset \ldots \supset \tau_{n-1} \supset  \tau_n$, then $D$ is carried by $\tau_{n-1}$.

\end{corollary}

\vskip10pt

\begin{proof}
Any curve $D$ which bounds an essential disk in $V_\cal D$ must  either belong to $\cal D$, or 
else  contain a wave with respect to $\cal D$  (see Remark \ref{otherdisks1}). Hence  Lemma 
\ref{carriedwaves} shows that the curve $D$ is $(n-1)$-gregarious on $\tau$.

\end{proof}

\vskip10pt

\section{Transverse train tracks}\label{transeverse}

\vskip10pt

We will now consider simultaneously two train tracks $\tau_1$ and $\tau_2$ on $\Sigma$, and for simplicity we assume that both are maximal. In order to work with them, they have to be placed in a ``tight position'' with respect to each other. To control this, we observe that after a suitable isotopy we may assume that the closure of every complementary component of $\tau_1 \cup \tau_2$ is an {\it $m$-gon}, with vertices that are either cusp points on $\partial \tau_1$ or $\partial \tau_2$, or they are {\it corners}, i.e. points that belong to the intersection $\partial\tau_1 \cap \partial\tau_2$. Hence the sides 
of these $m$-gons are arcs from $\partial \tau_1$ or from $\partial \tau_2$ that do not contain any cusp point  in their interior.

\begin{definition} \label{transverse}  \rm 
Two maximal 
train tracks $\tau_1, \tau_2 \subset \Sigma $ will be called  {\it transverse}  if the following conditions hold:

 \begin{enumerate}

\item Each connected component of  $\tau_1 \cap \tau_2$, with the two inherited $I$-fiberings, is homemorphic to the standard square with horizontal and vertical interval fibers.

In particular, every intersection arc of $\partial \tau_i$ with $\tau_j$, for $\{i, j\} = \{1, 2\}$, is a fiber in the $I$-fibering of $\tau_j$. Furthermore the intersection $\tau_1 \cap \tau_2$ does not contain any of the cusp points on either $\partial \tau_1$ or $\partial \tau_2$.

\item Every complementary component of $\tau_1 \cup \tau_2$ in $\Sigma$ is an $m$-gon (as defined above), with $m \geq 3$.

\end{enumerate}

\end{definition}

\vskip10pt

\begin{proposition} \label{minimalintersections}
Let $\tau_1, \tau_2 \subset \Sigma $ be maximal  transverse train tracks on $\Sigma$. 
If ${\cal L}_1, {\cal L}_2$  are laminations covering $\tau_1$ and  $\tau_2$ respectively, then 
${\cal L}_1$ and ${\cal L}_2$ intersect minimally.
\end{proposition}

\vskip10pt

\begin{proof}
If the minimal intersection number is not realized by ${\cal L}_1$ and ${\cal L}_2$ then there are leaves $l_1 \in {\cal L}_1$ and $l_2 \in {\cal L}_2$ which have a pair of intersection points that can be cancelled out by an isotopy. Hence there is a complementary region of $l_1 \cup l_2$  in $\Sigma$ that is a bigon.
Either this bigon is innermost  among all such bigons,  or we can find an innermost one by passing to other leaves  $l'_1, l'_2$. 

It follows that the segments of $l'_1$ and $l'_2$ which form the bigon are outermost (on the side of the bigon) on the branches of the train tracks $\tau_1$ and $\tau_2$ respectively, which carry these segments.  Since by assumption each $\cal L_i$ covers $\tau_i$, it follows that the bigon meets 
$\tau_1 \cup\tau_2$ only in a  collar around its boundary.  Hence, the remainder of the bigon is a complementary component of  $\tau_1 \cup\tau_2$ which inherits the bigon shape. This contradicts  condition (3) of Definition   \ref{transverse} and thus  our assumption that the train tracks are transverse.

\end{proof}

Since two maximal train tracks on $\Sigma$ can never be disjoint,  one obtains directly from the above proposition: 

\begin{corollary}\label{notcarried}  Let  $\tau_1, \tau_2 \subset \Sigma $  be  maximal 
transverse train tracks, and suppose that  ${\cal L}$ is a lamination that can be carried by both  
$\tau_1$ and $\tau_2$. Then $\cal L$ can not cover both train tracks.

\qed

\end{corollary}

\vskip10pt

\begin{proposition} \label{new-distance-lemma}
Let $\tau$ and $\tau'$ be maximal transverse  tight train tracks, and let $n \geq 1$ be any integer. 
Let  $\cal D \in \cal{CDS}(\tau)$ be $n$-gregarious with respect to $\tau$, and let 
$\cal E \in \cal{CDS}(\tau')$  be $2$-gregarious with respect to $\tau'$. Furthermore assume that all 
of the complementary components of $\cal D$ and  $\cal E$ in $\Sigma$ have $\Theta$-graph shape, 
with respect to $\tau$ and $\tau'$ respectively. 

Then  the handlebodies $V_\cal D$ and $W_\cal E$ determined by $\cal D$ and $\cal E$ respectively 
satisfy:
$$d(V_{\cal D}, W_{\cal E}) \geq n -1$$

\end{proposition}

\vskip10pt

\begin{proof}
We apply Corollary \ref{pre-distance} to obtain that any meridian curve $D$ for $V_\cal D$ is 
$(n-1)$-gregarious on $\tau$. Similarly, any meridian curve $E$ for $W_\cal E$ is 1-gregarious on 
$\tau'$.

Thus, by paragraph \ref{carried-parallel} (e) of subsection \ref{basics-on-tts}, the curve $E$ covers $\tau'$.  But $\tau'$ is transverse to $\tau$, so that by Corollary \ref{notcarried} the curve $E$ can not cover $\tau$.  Hence, again by paragraph \ref{carried-parallel} (e) of subsection \ref{basics-on-tts}, 
$E$ is not carried by any train track $\tau_1$ derived from $\tau$. Thus we can now apply Proposition \ref{towerimpliesdistance} to conclude that  $d(D, E) \geq n-1$. This shows that
$d( V_{\cal D}, W_{\cal E}) \geq n-1$, as claimed.

\end{proof}

\vskip10pt

\begin{lemma} \label{smallnbds} 
(a) Let $U \subset \cal{PML}(\Sigma)$ be any non-empty open set.
Then there is a maximal tight train track $\tau \subset \Sigma$ such that $U$ contains $\cal{PML}(\tau)$.

\vskip5pt

\noindent 
(b)  
Every non-empty open set  $\widehat U  \subset \cal{PLM}(\Sigma)^2$ contains the product 
$\cal {PML}(\tau) \times \cal {PML}(\tau')$, for some  transverse maximal tight train tacks $\tau$ and $ \tau'$.

\end{lemma}

\vskip10pt

\begin{proof}  (a) The set of  uniquely ergodic minimal-maximal laminations is dense in 
$\cal{PML}(\Sigma)$ (see paragraph  \ref{laminations-basics} (c) of subsection \ref{basics-on-tts}),  
so that $U$ contains such a lamination $\cal L$.  We choose, at random, a hyperbolic structure on 
$\Sigma$, and  isotope the lamination  $\cal L$ into geodesic position.  Let $\cal E$ be any complete decomposing system of $\Sigma$,  also assumed  to be in geodesic position.

Consider a complementary component $\Delta$ of $\cal L$, which by the minimal-maximal condition on $\cal L$  is  an ideal hyperbolic triangle embedded into $\Sigma$. The (geodesic) curves of  $\cal E$ cut through $\Delta$, always entering and exiting $\Delta$ through two distinct (geodesic) boundary 
sides.  Notice that in the direction of each of the cusps of $\Delta$ there are infinitely many intersection segments of $\cal E \cap \Delta$, since $\cal E$ is a complete decomposing system and  since 
no  infinite  half-leaf of $\cal L$ can stay within any of the complementary pairs-of-pants,  as 
$\cal L$ is minimal-maximal.

For any $\epsilon > 0$ we can consider the subsurface $\Delta_\epsilon \subset \Delta$ which is given by all points of distance $\geq \epsilon$ from the boundary $\partial \Delta$. The subsurfaces $\Delta_\epsilon$ are triangles with almost geodesic boundary edges (for small values of $\epsilon$), and by a proper variation of  $\epsilon$ we can force a cusp point of its boundary to lie on one of the intersection arcs of  $\cal E \cap \Delta$. We do this for each of the three cusp points individually, and vary $\epsilon$ accordingly  in a continuous fashion  along the triangle sides, to obtain a pseudo-geodesic hyperbolic triangle $\Delta'$ which has all three cusp points on $\cal E$.

Once the pseudo-hyperbolic triangles $\Delta'_i$ are obtained for  all of the complementary components of $\cal L$,  we  fiber the complement of the union of $\cal E$ and of all $\Delta'_i$ by small geodesic arcs which run ``parallel'' to the intersection arcs of $\cal E$ with $\Sigma \smallsetminus (\cup \Delta'_i)$, so that the closure 
of $\Sigma \smallsetminus (\cup \Delta'_i)$ becomes indeed a maximal train track $\tau$ which is tight with respect to $\cal E$, and $\tau$ carries $\cal L$. 

Now, if all of the parameters $\epsilon > 0$ in the above construction have been chosen small enough, the resulting train track $\tau$ carries only  laminations $\cal L'$ which, when isotoped into  geodesic position, intersect $\cal E$ in points that are very close to intersection points of $\cal E \cap \cal L$, and the intersection angles of $\cal E$ with $\cal L'$ will be very close to the corresponding angles with $\cal L$. Hence $\cal L'$ is very close to $\cal L$ in $\cal{PLM}(\Sigma)$.  This shows that for sufficiently small $\epsilon$ the set $\cal{PML}(\tau)$ will be contained in $U$.

\vskip5pt

\noindent (b) 
By the denseness of the uniquely ergodic minimal-maximal laminations (see  paragraph 
\ref{laminations-basics} (c) of subsection \ref{basics-on-tts}) we can find a pair  $(\cal L, \cal L')$ of such laminations in $\widehat U$ which in addition satisfies $\cal L \neq \cal L'$. Hence, if we perform  the same procedure as in part (a) with respect to the same hyperbolic structure on $\Sigma$, for both laminations,  for sufficiently small $\epsilon$-parameters the resulting train tracks $\tau, \tau'$ will be tight and maximal as before, and  $\cal{PML}(\tau) \times \cal{PML}(\tau') \subset U \times U$. In addition,  up to  making $\epsilon$ even smaller, the train tracks  $\tau$ and $ \tau'$ will be transverse with respect to each other.

\end{proof}

\vskip20pt
 
 \section{Density of $SNOW$}

\vskip10pt

Let $\tau$ be a maximal train track on $\Sigma$, and let $\cal D$ be a complete decomposing system carried by $\tau$.  

\begin{definition}\label{calm}\rm
We say that $\cal D$  {\it is calm with respect to $\tau$} (or simply $\cal D$ {\it is calm}) if none of the transverse $I$-fibers of $\tau$ contains a subarc that is a wave with respect to $\cal D$. The set of such calm complete decomposing systems carried by $\tau$ will be denoted by $\cal{CCDS}(\tau)$.

\end{definition}

\vskip10pt

The proof of the following lemma is elementary and thus left to the reader. The reader should compare its statement to that of Remark \ref{theta-waves} which is also about waves on calm complete decomposing systems.


\vskip10pt

\begin{lemma} \label{Theta-equals-calm}
Let $\tau$ be a maximal train track on $\Sigma$. Then any complete decomposing system $\cal D$ carried by $\tau$ is calm if and only if every complementary pair-of-pants has $\Theta$-graph shape.

\qed

\end{lemma}

The following is a crucial ingredient for the main result of this paper:

\vskip10pt
 
\begin{proposition} \label{no-wave-tt}
Every maximal train track $\tau \subset \Sigma$ carries some complete decomposing system $\cal D$ 
which is calm.

\end{proposition} 

\vskip5pt

\begin{proof}
Since the train track $\tau$ is maximal, all of its complementary regions  in $\Sigma$ are triangles. Choose an arbitrary partition of these  regions into pairs $\Delta_i, \Delta'_i$.  Consider the pair $\Delta_1$ and 
$\Delta'_1$ of complementary regions, and denote by  $c_1, c_2, c_3$ the cusps of $\Delta_1$ and by 
$c'_1, c'_2, c'_3$ those of $\Delta'_1$. We will  build a $\Theta$-shaped pair-of-pants out of $\Delta_1$ and $\Delta'_1$ by constructing three paths  $\gamma_{i,j}$ as above, which are carried by $\tau$ and connect the cusp $c_i$ to the cusp $c'_j$.

We will show below that, using an appropriate lamination $\cal L$ carried by $\tau$, one can find such paths 
$\gamma_{i,j}$ by exhibiting finite subsegments $l_{i,j}$ of some leaf $l$ of $\cal L$. This subsegment $l_{i,j}$ has endpoints  $R, S$ on the same (singular) $I$-fibers $I_1$ and $I_2$ of $\tau$ as the two cusps $c_i$ and $c'_j$. One then moves the ends of $l_{i,j}$ by an $I$-fiber preserving isotopy so that it reaches a position 
$l_{i,j}^*$ with endpoints on $c_i$ and $c'_j$.  The problem, of course, is that the isotopy can only be  performed if $l_{i,j}$ is {\it properly chosen}: this means that the interior of $l_{i,j}$  does not intersect either of the subsegments $[R, c_i] \subset I_1$ or $[S, c'_j] \subset I_2$, i.e. the subsegments on the two singular $I$-fibers which are bounded by the cusps and the endpoints of $l_{i,j}$. The following procedure avoids this problem:

\vskip5pt

\begin{enumerate}

\item[(1)]   By the denseness of the set of uniquely ergodic minimal-maximal laminations (see paragraph 
\ref{laminations-basics} (c) of subsection \ref{basics-on-tts}) there exists such a lamination  $\cal L$ that covers $\tau$.  For any sides $\delta$ of  $\Delta_1$ and $\delta'$ of $\Delta'_1$ there are boundary leaves of $\cal L$ which run parallel along all of  $\delta$ or $\delta'$.  Since in minimal-maximal laminations every leaf is dense,  any non-boundary leaf $l$ of $\cal L$ will contain  a finite segment $l_0$ which starts with a subsegment that runs parallel to $\delta$ and ends in a subsegment that runs parallel to $\delta'$, 
see Figure 1.  By possibly passing to a smaller subsegment we can assume that $l_0$ is a minimal such finite segment of $l$, i.e., it contains no proper subsegment with the same property. It follows  that the subsegment $l_{1, 1}$ of $l_0$, obtained from $l_0$ by cutting off the initial and the terminal subarcs parallel to $\delta$ and to $\delta'$ respectively, connects the singular fiber of a cusp of 
$\Delta_1$ (say  $c_1$) to the singular fiber of a cusp of $\Delta'_1$ (say $c'_1$). Note that $l_{1, 1}$ is properly chosen in the above defined sense, by our minimality requirement on $l_{1, 1}$. Isotope $l_{1,1}$ as above along the fibers to get a path $l_{1,1}^*$ which connects $c_1$ to $c'_1$ and set  
$\gamma_{1,1} = l_{1,1}^*$.

\vskip10pt

\begin{figure}[ht]
{\epsfxsize = 4.5 in \centerline{\epsfbox{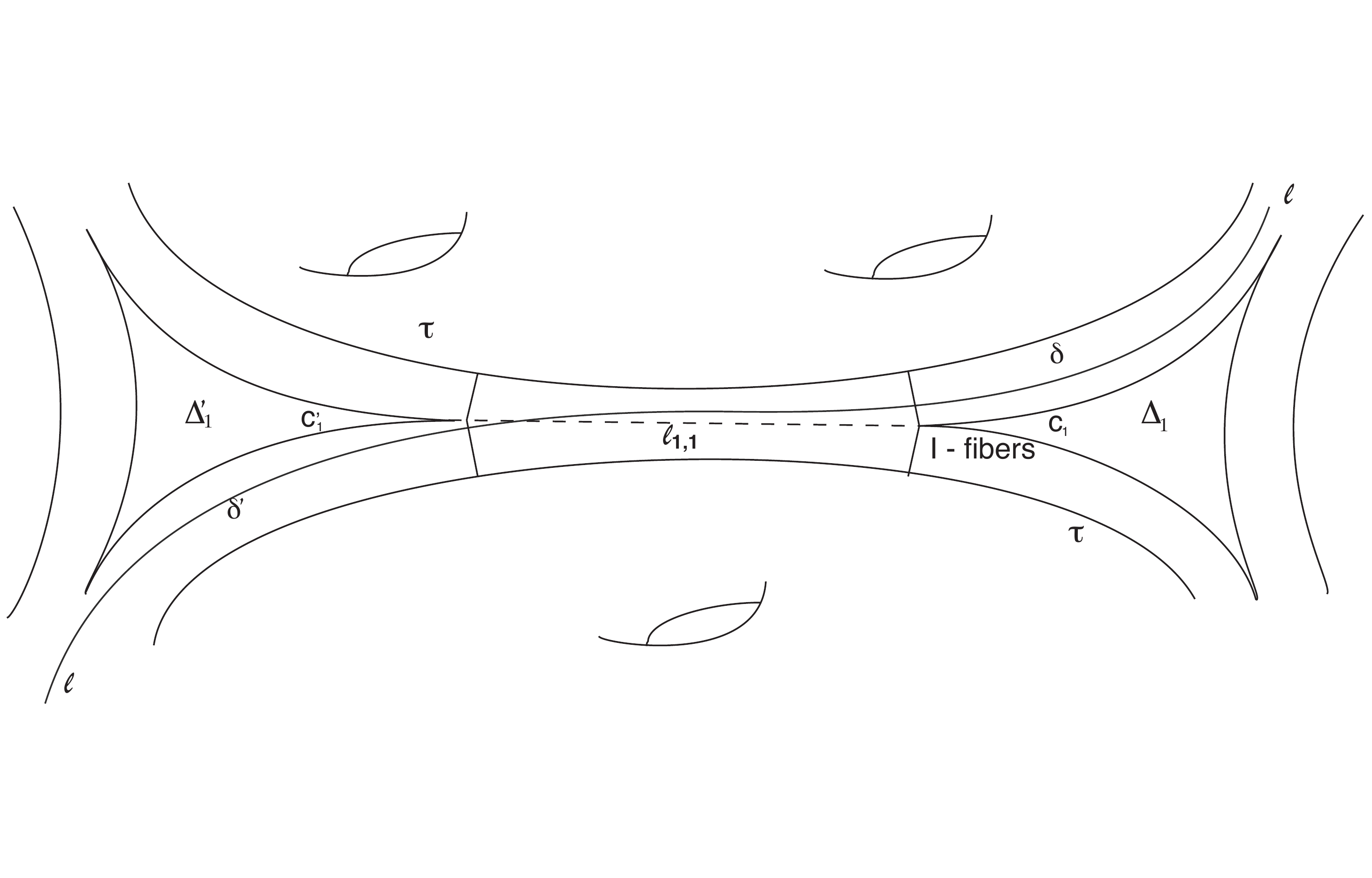}}}
\caption{}
\label{Figure 1}
\end{figure}

\vskip10pt

\item[(2)] 
Split $\tau$ open along $\gamma_{1,1}$ to get a train track $\tau'$, and consider some  uniquely ergodic minimal-maximal lamination (again called $\cal L$) that fills $\tau'$.  Let $\delta$ and $\delta'$ be the sides of $\Delta_1$ and  $\Delta'_1$ respectively which are situated opposite the cusps $c_1$ and  $c'_1$. 
Now proceed precisely in the same manner as  above in part (1).  As before, we obtain a simple arc  $l_{2, 2}^*$ carried by $\tau'$ that connects $c_2$ to $c'_2$. We set   $\gamma_{2, 2} = l_{2,2}^*$. 

\end{enumerate}

\vskip10pt

As before we split $\tau'$ open along $\gamma_{2,2}$ to obtain a train track $\tau''$, and again we are 
looking for a simple arc on some leaf of a lamination $\cal L$ that is uniquely ergodic minimal-maximal 
and which fills  $\tau''$, in order to connect $c_3$ to $c'_3$. At this point, however, the problem of  finding a properly chosen segment   $l^*_{3,3}$, becomes more delicate:

\begin{enumerate}

\item[(3)] We first chose an arc $d$ on the boundary of $\tau''$ that starts at $c_3$, runs along 
a side of $\Delta_1$, then proceeds along either $\gamma_{1,1}$ or $\gamma_{2,2}$, and finally runs along a side of  $\Delta'_1$, to end at $c'_3$.

Since $\cal L$ is  minimal-maximal and fills $\tau''$, we can find a non-boundary leaf $l$ in $\cal L$ which contains a finite  segment $d_1$ that runs parallel to $d$. We now consider an arc on $l$ which starts at an endpoint of $d_1$.  Since every leaf of $\cal L$ is dense in $\cal L$,  any sufficiently long such arc contains a segment $d_2$ of $l$ that is parallel to $d$ and $d_1$ and lies between them. As before, we assume that the arc $l$ is {\it minimal}, i.e. it doesn't contain a proper subarc which has the same properties as used above to define $l$.

If $d_1$ and $d_2$ are traversed in the same direction when travelling along $l$, then the segment on $l$ between $d_1$ and $d_2$ is the desired properly chosen arc $l_{3,3}$.  (See Figure 2.)

If not,  consider an arc on $l$ with the same initial point but which runs in the opposite direction. As above, if the arc is sufficiently long, it will contain a subarc $d_3$ that runs between $d$ and $d_2$ 
(where, as before, we assume that $d_3$ is the first occurrence of such a subarc when travelling along $l$). 
Again, if $l$ traverses $d_1$ and $d_3$ in the same direction, the subpath between $d_1$ and $d_3$ is the desired properly chosen subsegment $l_{3,3}$  (note that $d_2$ is not part of this subpath !). If the directions are  opposite, then the subpath of $l$ which connects $d_2$ to $d_3$ is the desired properly chosen segment  $l_{3,3}$. As before, set $\gamma_{3,3} = l_{3,3}^*$.  (See Figure 3.)

\end{enumerate}

\begin{figure}[ht]
{\epsfxsize = 5 in \centerline{\epsfbox{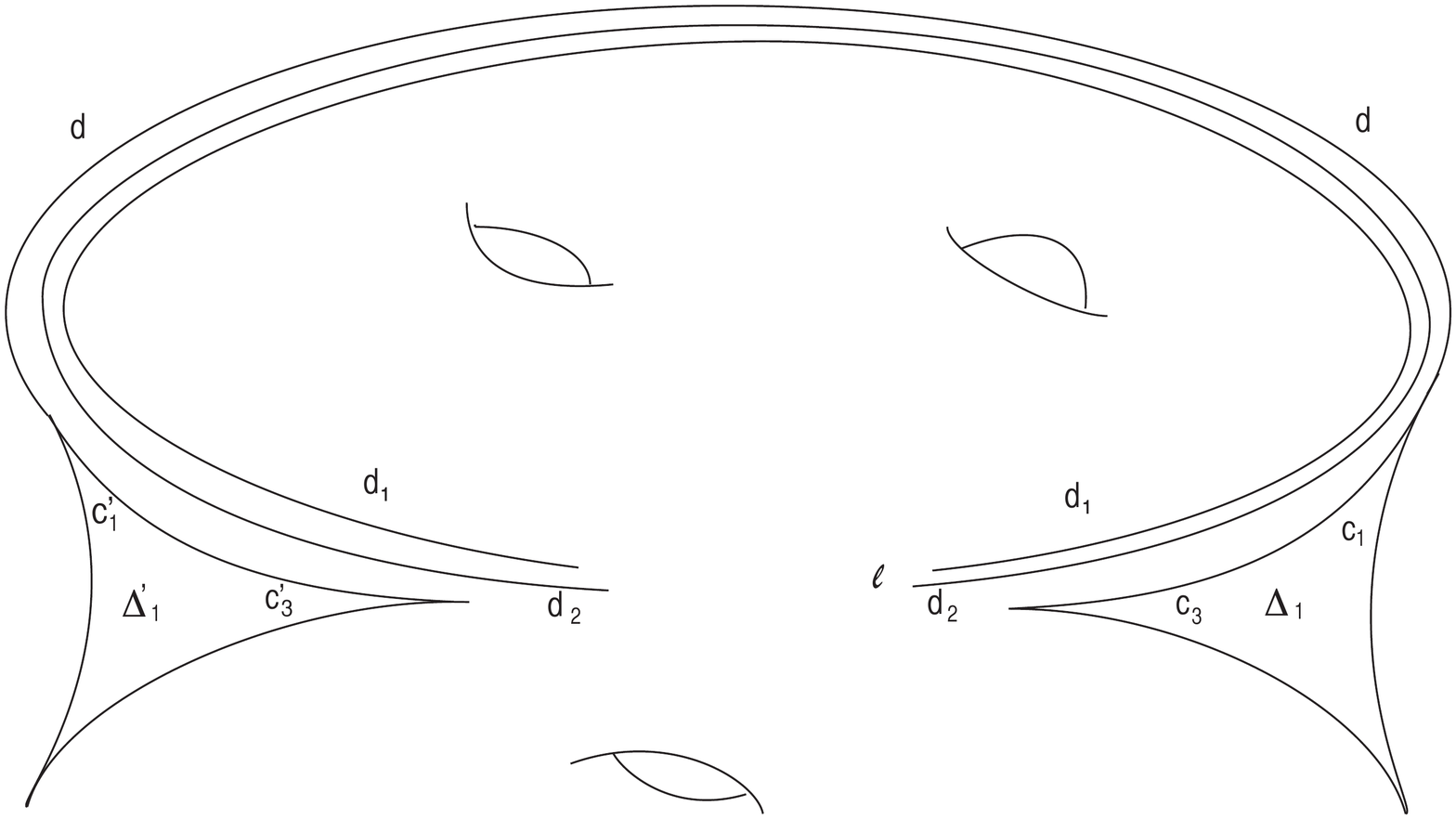}}}
\caption{}
\label{Figure 2}
\end{figure}

\begin{figure}[ht]
{\epsfxsize = 5 in \centerline{\epsfbox{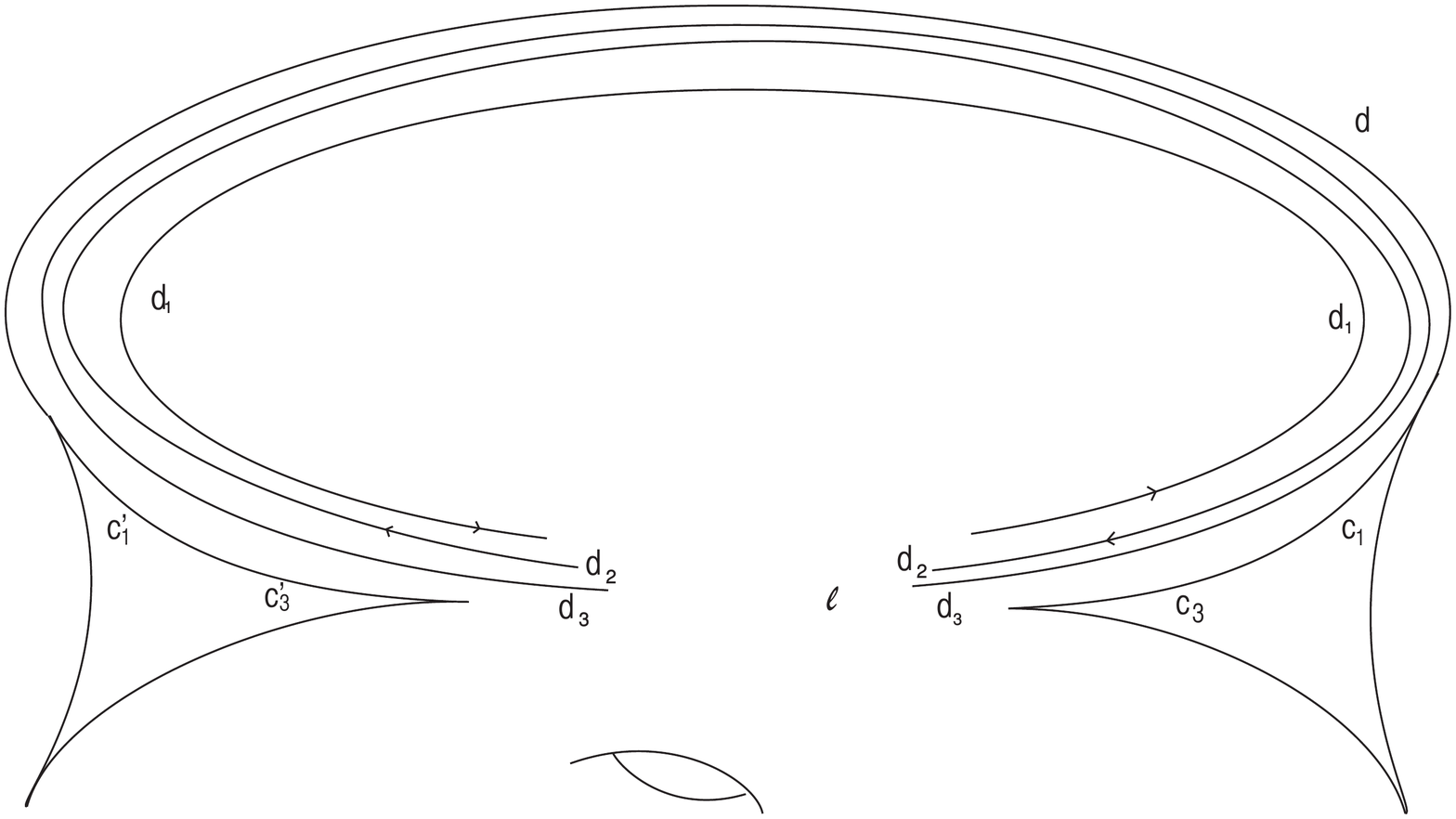}}}
\caption{}
\label{Figure 3}
\end{figure}

\vskip10pt

We now split $\tau''$ open along $\gamma_{3,3}$, to obtain a train track $\tau_1$, which has an $I$-fibering 
by subarcs of the $I$-fibers of $\tau$, and which has as complementary surfaces one pair-of-pants (which contains $\Delta_1$ and $\Delta'_1$ as subsurfaces),  while all of the other complementary subsurfaces 
coincide with the complementary subsurfaces $\Delta_i$ and $\Delta'_i$ ($i \neq 1$) of the original train track $\tau$. We now repeat the same procedure as before, for $\tau_1$ instead of $\tau$ and $\Delta_2, \Delta'_2$ instead of $\Delta_1, \Delta'_1$.

We proceed iteratively until none of the $\Delta_i, \Delta'_i$ are left: In this process, we obtain more and more complementary components of the resulting train tracks $\tau_i$ which are pairs-of-pants, and less and less triangles $\Delta_j, \Delta'_j$. It is possible that in this process the train track decomposes into more than one connected component, so that we are perhaps forced to reorganize the pairing of the triangles.  However, by an elementary Euler characteristic argument, if the boundary curve of some pair-of-pants separates the surface $\Sigma$, then on each side there must be an even number of triangles, so that we can continue the iteration procedure with $\tau_i$ being a connected component of the train track obtained after the (i - 1)-th iteration step.

After the last step of our iteration none of the triangles is left over, so that the union of all connected components of the train track doesn't have any cusps in its boundary:  It is hence a disjoint union of fibered annuli $A_i$. 

Every connected component in the complement of these annuli is a pair-of-pants which by construction has 
$\Theta$-shape with respect to the $I$-fibering of the original train track $\tau$.  We can now contract each 
of the annuli $A_i$  by an $I$-fiber-preserving isotopy until the two boundary curves coincide. The resulting curve system $\cal D$ is precisely the desired complete decomposing system: It is carried by $\tau$, and each complementary pair-of-pants has  $\Theta$-graph shape.  It follows from Lemma \ref{Theta-equals-calm} that $\cal D$ is calm 
with respect to $\tau$.

\end{proof}

\vskip10pt

\begin{corollary}\label{densecalm} For every maximal train track $\tau \subset \Sigma$  the set  
$\cal{CCDS}(\tau)$ of calm complete decomposing systems carried by $\tau$ is dense  
$\cal{PML}(\tau)$.

\vskip10pt

\end{corollary}

\begin{proof} 
By Lemma \ref{smallnbds} (a) for every open set $U$ in $\cal{PML}(\tau)$ there is a maximal train track 
$\tau'$ such that $U$ contains $\cal{PML}(\tau')$. Since the set of uniquely ergodic minimal-maximal laminations is dense in $\cal{PML}(\Sigma)$ (see paragraph \ref{laminations-basics} (c) of subsection \ref{basics-on-tts}), there is such a lamination $\cal L$ in $\cal{PML}(\tau) \subset U$. Since $\cal L$ is maximal, it covers $\tau'$. Hence we can split $\tau'$ following the leaves of $\cal L$, until the $I$-fibers of $\tau'$ are small enough so that they can be isotoped simultaneously onto $I$-fibers of $\tau$, thus showing that $\tau'$ is carried by $\tau$. Now we apply Proposition \ref{no-wave-tt} to obtain a complete decomposing system $\cal D$ which is calm with respect to $\tau'$, and hence also with respect to 
$\tau$.

\end{proof}

\vskip10pt

\begin{lemma} \label{doubleNW-implies-SNOW}
Let $\tau$ and $\tau'$ be maximal transverse train tracks on $\Sigma$, and let $\cal D$ and $\cal D'$ be complete decomposing systems carried by $\tau$ and $\tau'$ respectively. Let $\tau_1$ and $\tau'_1$
be the $1$-derived train tracks with respect to $\cal D$ and $\cal D'$.  We have:

\begin{enumerate}

\item[(i)] If $\cal D$ and $\cal D'$ are calm with respect to $\tau$ and $\tau'$ respectively, then they are $SNOW$.

\item[(ii)] If $\cal D$ and $\cal D'$ are carried by $\tau_1$ and $\tau'_1$ and are $SNOW$, then they are
calm with respect to $\tau$ and $\tau'$ respectively.

\end{enumerate}

\end{lemma}

\vskip10pt

\begin{proof} (i) 
If there was a wave $\omega$, say on $\cal D$ with respect to $\cal D'$, then, by transversality of $\tau$ and $\tau'$,  there are precisely two possibilities:

Either $\omega$ is contained in one of the ``doubly-fibered'' crossing of the two train tracks, thus giving rise 
(see Figure 4)  to a wave with respect to $\cal D'$ on some of the transverse $I$-fibers of $\tau'$, in contradiction to the assumption that  $\cal D'$ is calm.

\begin{figure}[ht]
{\epsfxsize = 4.5 in \centerline{\epsfbox{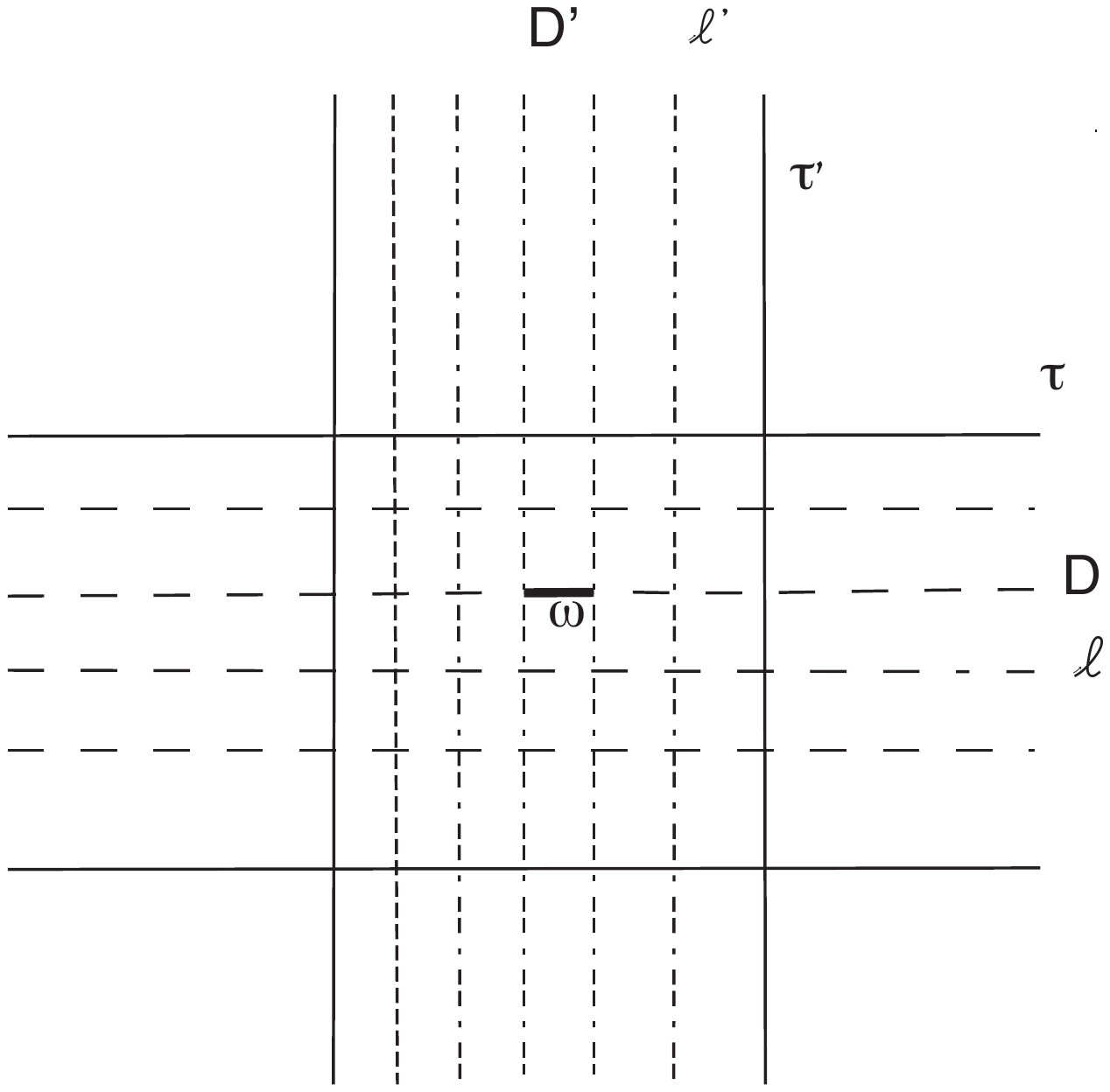}}}
\caption{The darkened arc $\omega$ is a wave on $D$ with respect to $D'$.}
\label{Figure 4}
\end{figure}

The other possibility is that the wave $\omega$ (up to small initial and terminal segments that belong to $\tau'$) is contained in one of the complementary components $\Delta$ of $\tau'$, so that $\omega$ connects two of the curves of $\cal D'$ that are outermost on $\tau'$.  But since all of the complementary components of $\tau'$ are triangles, $\Delta$  is divided by $\omega$ into a part that contains two cusps, and another part that contains a single cusp  (see Figure 5).   Isotope this wave (so that its endpoints move on $\cal D'$) in the direction of the ``single-cusp'', until it slides over that cusp, so that it will reach a position parallel to a segment of one of the transverse $I$-fibers of $\tau'$. Thus it will define a wave on such an $I$-fiber, giving the same contradiction as in the first case.

\vskip5pt

\begin{figure}[ht]
{\epsfxsize = 4.5 in \centerline{\epsfbox{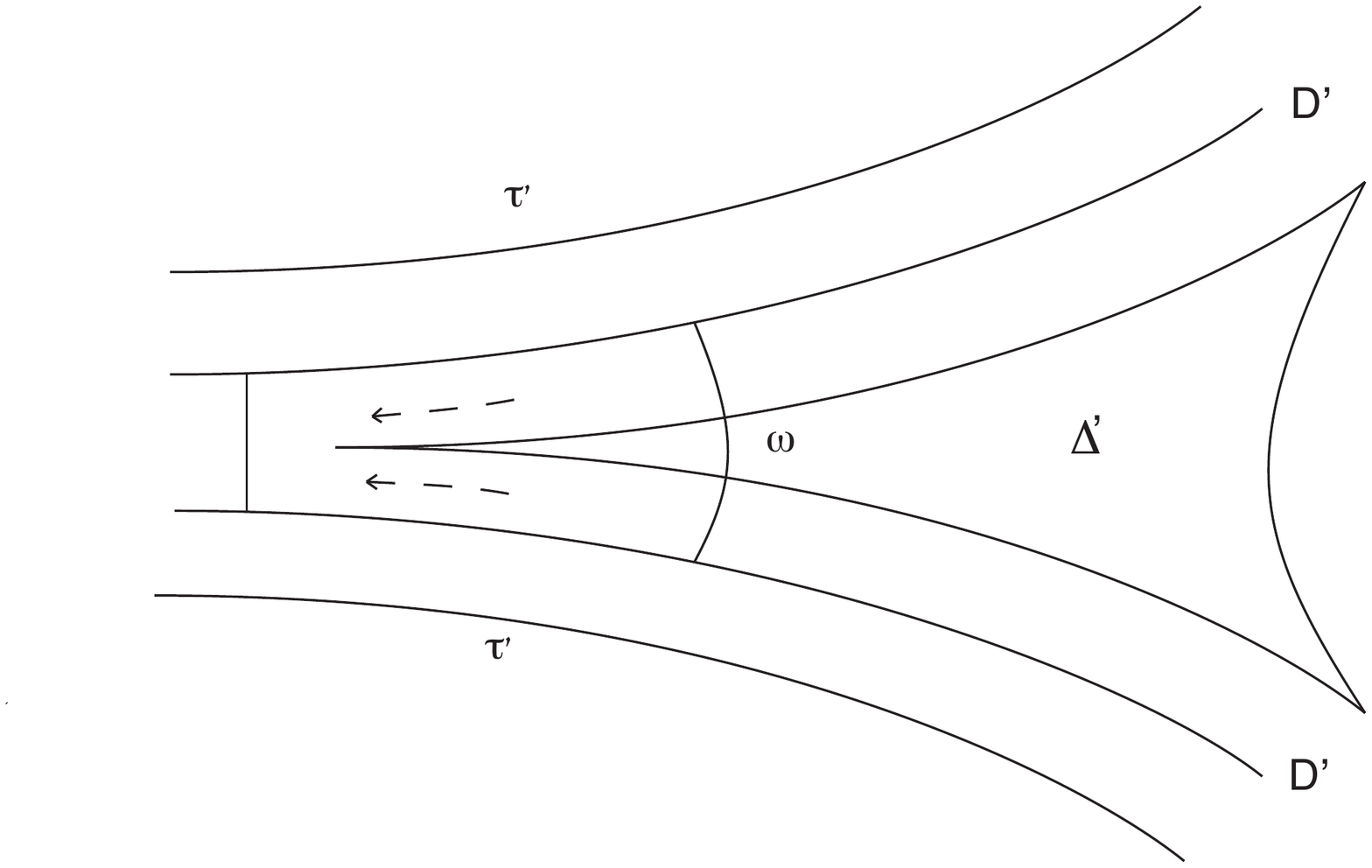}}}
\caption{}
\label{Figure 5}
\end{figure}

\vskip5pt

\noindent
(ii) Assume by way of contradiction that, say, ${\cal D}$ is not calm with respect to $\tau$. This means  that there
 is a transverse $I$-fiber of $\tau$ which contains  a wave $\omega$ with respect to ${\cal D}$. That is,  $\omega$ is  an arc with both end points on the same leaf $l \in {\cal D}$. Since  ${\cal D}$ is carried by the derived train track  $\tau_1$, the two endpoints of $\omega$ are contained in  two subarcs of $l$ which run parallel  (on $\tau$) 
to each other along a path $\alpha$ that covers $\tau$,  see \ref{carried-parallel} (e) of subsection 
\ref{basics-on-tts}. 

Hence $\omega$ can be isotoped in $\tau$ along the path $\alpha$, and during this isotopy its endpoints never  ``split'' across a cusp point of $\tau$. 

As  $\tau$ and $\tau'$ are transverse, the isotopy along $\alpha$ will eventually take $\omega$ into 
some ``doubly-fibered'' intersection rectangle of  $\tau \cap \tau'$, where $\omega$ becomes an arc which 
is parallel to a fiberof $\tau$. But that means that there is some arc on a leaf of ${\cal D}'$ which is parallel 
to $\omega$ and hence ${\cal D}'$ will contain a wave with respect to ${\cal D}$. So that ${\cal D}$ and 
${\cal D}'$ are not $SNOW$.

\end{proof}

\vskip10pt

\begin{proposition} \label{SNOW-is-dense}
The subset of pairs $({\cal D}, {\cal D'})$ that are SNOW is dense in $\cal {PML}(\Sigma)^2$.

\end{proposition}

\vskip10pt

\begin{proof}
Every lamination on $\Sigma$ is carried by some maximal train track $\tau_i \subset \Sigma$, so that $\cal{PML}(\Sigma)$  is the union over a suitable (countable) family of subspaces $\cal{PML}(\tau_i) \subset \cal{PML}(\Sigma)$. In each $\cal{PML}(\tau_i)$ the set $\cal{CCDS}(\tau_i)$ is dense, by Corollary \ref{densecalm}.

Hence the set $\cal{CCDS}(\tau_i) \times \cal{CCDS}(\tau_j)$ is dense in $\cal{PML}(\tau_i) \times \cal{PML}(\tau_j)$, for all pairs of indices $i,j$, and the union of them is equal to $\cal{PML}(\Sigma)^2$.
 By Lemma \ref{doubleNW-implies-SNOW} (a) any pair of complete decomposing systems $(\cal D, \cal D') \in \cal{CCDS}(\tau_i) \times \cal{CCDS}(\tau_j)$ is $SNOW$. Hence the set ${SNOW}(\Sigma)$ is dense in $\cal {PML}(\Sigma)^2$.

\end{proof}

\vskip10pt

\section{Open and dense subsets of $SNOW(\Sigma)$}

\vskip10pt

We consider maximal train tracks $\tau$ and $\tau'$ on $\Sigma$, and we assume that they are transverse. Given two complete decomposing systems ${\cal D}, {\cal E} \in \cal{CDS}(\Sigma)$.  The handlebodies they determine are  denoted by $V_{\cal D}, W_ {\cal E} $ respectively. 

\vskip10pt

\begin{definition} \label{ngregarious+}\rm
We denote the subset  of  $\mathcal{PML}(\tau)$  given by all   $n$-gregarious laminations (see Definition 
\ref{ngregarious}) by ${\bf G^n}(\tau)$, and  define $\bf{G}^{n}\bf{P}(\tau)$ $= {\bf   G}^n(\tau) \cap {\bf P}(\tau)$.  Here ${\bf P}(\tau)$ denotes the set of all laminations that define positive weight on each $I$-fiber of the train 
track $\tau$.

\end{definition}

\vskip10pt

\begin{proposition} \label{open-and-full-measure}
Let $\tau$ and $\tau'$ be maximal transverse train tracks, and let $n \geq 1$ be an integer. Then one has:
\begin{enumerate}

\item[(a)] The subset 
$\bf{G}^{n}\bf{P}(\tau) \times \bf{G^2P}(\tau')$  is open and of full measure in 
$\mathcal{PML}(\tau) \times \mathcal{PML}(\tau')$.

\vskip5pt

\item [(b)] The subset $\bf{G}^{n}\bf{P}_\mathcal{CCDS}(\tau, \tau') := $
$$(\mathcal{CCDS}(\tau) \times \mathcal{CCDS}(\tau')) \cap (\bf{G}^{n}\bf{P}(\tau) \times \bf{G^2P}(\tau'))$$
 is dense in  $\bf{G}^{n}\bf{P}(\tau) \times \bf{G^2P}(\tau')$.
 
 \vskip5pt

\item[(c)]Every pair $({\cal D}, {\cal E}) \in \bf{G}^{n}\bf{P}_\mathcal{CCDS}(\tau, \tau')$
satisfies: $$d(V_{\cal D}, W_{\cal E}) \geq n -1$$

\end{enumerate}

\end{proposition}

\vskip5pt

\begin{proof} (a) This has been proved as Propositions 3.10 and 3.11 in \cite{LM2}.
\vskip5pt

\noindent (b)  This follows directly from the fact that $\cal {CCDS}(\tau)$ and $\cal {CCDS}(\tau')$ are dense in $\cal {PML}(\tau)$ and $\cal {PML}(\tau')$ respectively (see Corollary \ref{densecalm}), since the subsets $\bf{G}^{n}\bf{P}(\tau)$ and $\bf{G^2P}(\tau')$ are open.

\vskip5pt

\noindent (c)  This statement is a direct consequence of Proposition \ref{new-distance-lemma}.

\end{proof}

\vskip5pt

We also recall the following:

\begin{remark} \label{open-dense} \rm
(a) The union of all products  $\bf{P}(\tau) \times \bf{P}(\tau')$, over all  pairs $\tau, \tau'$ of  tight maximal 
transverse train  tracks on $\Sigma$, is an open and dense set in ${\cal PML}(\Sigma)^2$. 

This is a direct consequence of Lemma \ref{smallnbds} (b) and the fact that the sets $\bf{P}(\tau)$ are open by definition.

\vskip5pt
\noindent
(b) In particular, the above union is of full measure in $\cal{PML}(\Sigma)^2$. 

This follows directly from statement (a) and from the fact that our measure is a Lebesgue measure on the manifold $\cal{PML}(\Sigma)^2$. 

\end{remark}

\begin{remark} \label{gregarious-calm}\rm
In the proof of the next theorem we need to compare, for any pair of transverse maximal train tracks 
$\tau$ and $\tau'$, the set  $SNOW(\Sigma) \cap (\cal{PLM}(\tau) \times \cal{PLM}(\tau'))$ with the set 
$\cal{CCDS}(\tau) \times \cal{CCSD}(\tau')$. From Proposition \ref{doubleNW-implies-SNOW} one doesn't quite get equality of the two sets, since in part (ii) of this proposition one needs the additional property that the curve systems $\cal D$ and $\cal D'$ are carried by once derived train tracks 
$\tau_1 \subset \tau$ and $\tau'_1 \subset \tau'$.

However, if we restrict our attention to those systems that are in addition 1-gregarious, this extra condition is always satisfied, so that Proposition \ref{doubleNW-implies-SNOW} proves:
$$SNOW(\Sigma) \cap (\bf{G}^{1}\bf{P}(\tau) \times \bf{G^1P}(\tau')) =  $$
$$(\cal{CCDS}(\tau) \times \cal{CCSD}(\tau'))\cap (\bf{G}^{1}\bf{P}(\tau) \times \bf{G^1P}(\tau'))$$

\end{remark}

\vskip10pt

\begin{theorem} 
\label{large_distance_snows_are_generic}
For any integer $n \geq 1$, the set $SNOW_{n - 1}(\Sigma)$ is generic in the set 
$SNOW(\Sigma)$.
\end{theorem}

\vskip5pt

\begin{proof}
This proof  is 
an application of the genericity criterium stated in Lemma \ref{equivalentgeneric} (1), with the notation 
used there  specified  as follows:
\vskip4pt
\begin{enumerate}
\item $X = {\cal PML}(\Sigma)^2$
\vskip4pt
\item $Y = SNOW(\Sigma)$
\vskip4pt
\item $A =  SNOW_{n-1}(\Sigma)$
\vskip4pt
\item $Z = \cup \, \bf{G}^{n}\bf{P}(\tau) \times \bf{G^2P}(\tau')$, where the union is taken over all pairs of transverse maximal  tight train tracks $\tau, \tau'$ on $\Sigma$.

\end{enumerate}

\noindent
According to Lemma \ref{equivalentgeneric} (1) it suffices to show that: 

\vskip5pt
\noindent
\quad (a)  $Z$ is contained in the closure of $SNOW_{n - 1}(\Sigma)$,

\vskip5pt
\noindent
\quad (b) $Z$ is  open and  of full measure in ${\cal PML}(\Sigma)^2$,  and 

\vskip5pt
\noindent
\quad (c) that this set  $Z$  is disjoint from  $SNOW(\Sigma)  \smallsetminus SNOW_{n - 1}(\Sigma)$.  

\vskip10pt

\noindent
(a) By Lemma  \ref{doubleNW-implies-SNOW} (i) and Proposition \ref{open-and-full-measure} (c), 
for any two transverse maximal  tight train tracks $\tau, \tau'$ on $\Sigma$, every pair 
$({\cal D}, {\cal E}) \in \bf{G}^{n}\bf{P}_\mathcal{CCDS}(\tau, \tau')$ is $SNOW$ and 
$d(V_{\cal D},W_{\cal E}) \geq n - 1$.  We obtain:
$${\bf G}^{n}{\bf P}_\mathcal{CCDS}(\tau, \tau') \subset SNOW_{n - 1}(\Sigma) $$

By Proposition \ref{open-and-full-measure} (b) the set $\bf{G}^{n}\bf{P}_\mathcal{CCDS}(\tau, \tau')$ is 
dense in  $\bf{G}^{n}\bf{P}(\tau) \times \bf{G^2P}(\tau')$. Thus  $Z$  is contained in the closure of all  
$\bf{G}^{n}\bf{P}_\mathcal{CCDS}(\tau, \tau')$ and thus in the closure of $SNOW_{n - 1}(\Sigma)$.

\vskip5pt

\noindent
(b)  The set  $\bf{G}^{n}\bf{P}(\tau) \times \bf{G^2P}(\tau')$ is open and of full measure in
$\cal{PML}(\tau) \times \cal{PML }(\tau')$, by Proposition \ref{open-and-full-measure} (a). Hence this set is of full measure in the open subset ${\bf P}(\tau) \times {\bf P}(\tau') \subset \cal{PML}(\tau) \times \cal{PML }(\tau')$. By Remark \ref{open-dense} the union of all  ${\bf P}(\tau) \times {\bf P}(\tau')$ is open and of full measure in  $\cal{PML}(\Sigma)^2$. Hence the union  $Z$  of all  
$\bf{G}^{n}\bf{P}(\tau) \times \bf{G^2P}(\tau')$ is open  and of full measure in $\cal{PML}(\Sigma)^2$. 

\vskip5pt

\noindent
(c) The intersection of $SNOW(\Sigma)$ with the union  $Z$ 
of all $\bf{G}^{n}\bf{P}(\tau) \times \bf{G^2P}(\tau')$ is the union of the intersections of $SNOW(\Sigma)$ with any of the individual sets $\bf{G}^{n}\bf{P}(\tau) \times \bf{G^2P}(\tau')$. Consider therefore the set 
$SNOW(\Sigma) \cap (\bf{G}^{n}\bf{P}(\tau) \times \bf{G^2P}(\tau'))$.  This set is contained in the product set   $\bf{G}^{1}\bf{P}(\tau) \times \bf{G^1P}(\tau')$,  by definition. However, the intersection of 
$\bf{G}^{1}\bf{P}(\tau) \times \bf{G^1P}(\tau')$ with $SNOW(\Sigma)$ is precisely 
$$(\cal{CCDS}(\tau) \times \cal{CCSD}(\tau')) \cap (\bf{G}^{1}\bf{P}(\tau) \times \bf{G^1P}(\tau'))\, ,$$ 
by Remark \ref{gregarious-calm}. Hence we obtain
$$SNOW(\Sigma) \cap (\bf{G}^{n}\bf{P}(\tau) \times \bf{G^2P}(\tau'))  = $$
$$SNOW(\Sigma) \cap (\bf{G}^{1}\bf{P}(\tau) \times \bf{G^1P}(\tau')) \cap (\bf{G}^{n}\bf{P}(\tau) \times \bf{G^2P}(\tau'))  =$$
$$(\cal{CCDS}(\tau) \times \cal{CCSD}(\tau')) \cap (\bf{G}^{1}\bf{P}(\tau) \times \bf{G^1P}(\tau')) \cap (\bf{G}^{n}\bf{P}(\tau) \times \bf{G^2P}(\tau'))  =$$
$$(\cal{CCDS}(\tau) \times \cal{CCSD}(\tau')) \cap (\bf{G}^{n}\bf{P}(\tau) \times \bf{G^2P}(\tau'))  =$$
$$\bf{G}^{n}\bf{P}_\mathcal{CCDS}(\tau, \tau') \, ,$$
where the last equation is simply the definition of $ \bf{G}^{n}\bf{P}_\mathcal{CCDS}(\tau, \tau')$. 
We can now apply Proposition \ref{open-and-full-measure} (c) to obtain that this set is contained in 
$SNOW_{n - 1}(\Sigma)$. Thus  $Z$  is disjoint from  $SNOW(\Sigma)  \smallsetminus SNOW_{n - 1}(\Sigma)$.

\vskip5pt

This finishes the proof of the theorem.

\end{proof}

\vskip20pt 

\section{Generic Heegaard splittings have large distance}
\label{generic-H-splittings}

\vskip5pt
 In this section we study, for a  surface $\Sigma$ of a fixed genus $g \geq 2$, the set $\cal H^g$, 
 defined as follows: Recall that a handlebody set $\cal D(V) \subset \cal C(\Sigma)$ is given by some identification $\Sigma = \partial V$, where $V$ is a handlebody and $\cal D(V)$ consists precisely of those curves on $\Sigma$ that are identified by $\Sigma = \partial V$ with meridians of $V$. We define the set  $\cal {H}^g$  to consist of all pairs   $(\cal{D}(V), \cal{D}(W))$  of handlebody sets in $ \cal{C}(\Sigma)$. 
Of course, the identifications $\partial V = \Sigma = \partial W$ define a closed 3-manifold $M$ which contains $\Sigma$ as a Heegaard surface.

We further define $\cal {H}^g_n \subset \cal {H}^g $ to be the set of such  pairs which are of distance 
$\geq n$.

By Lemma \ref{hempellemma} each  irreducible Heegaard splitting admits a pair of complete decomposing systems  $(\cal D, \cal E)$  which  are in $SNOW(\Sigma)$. Hence there exist a map
$$\sigma_{\cal {H}}: \,  \cal H^g \,\, \, \to \, \, \, SNOW(\Sigma) \,  \subset \,  \cal{PML}(\Sigma)^2$$
such that $\sigma_\cal H (\cal{D}(V), \cal{D}(W))  = (\cal D, \cal E)$  satisfies $\cal D \in \cal D(V),\, \cal E \in \cal D(W) $.

\vskip10pt

There are many different possible choices  for the map $\sigma_\cal H$,  and  none of them may merit 
the attribute ``natural''. However, we point out that  the results of this section hold for any such map 
$\sigma_{\cal {H}}$.

As in the previous sections,  the natural Lebesgue measure class on $\cal{PML}(\Sigma)^2$ naturally defines generic subsets of  $\sigma_{\cal {H}}(\cal H^g)$.

\vskip10pt

\begin{corollary} 
\label{large_distance_HSs_are_generic}
For any integers $g \geq 2$ and $n \geq 0$ the set $\sigma_{\cal {H}}(\cal{H}^g_n)$ is generic in the set 
$\sigma_{\cal {H}}(\cal{H}^g)$.
\end{corollary}

\vskip5pt

Notice that there is a canonical bijection between the set $\cal S^g$ of marked Heegaard surfaces as introduced in subsection \ref{MandO}, and the set $\cal H^g$ studied here: The marking homeomorphism $\theta: \Sigma \to S_g$ of a marked Heegaard surface $S_g$ in a 3-manifold $M$, with complete meridian decomposing systems $\cal D$ and $\cal E$ for the two handlebody complements of $S_g$ in $M$, defines canonically a pair of handlebody sets $\cal D(V_{\theta^{-1}(\cal D)}, \cal D(V_{\theta^{-1}(\cal E)})$, and conversely.  Since the described canonical bijection between $\cal S^g$ and $\cal H^g$ is clearly distance preserving, 
the statement of Corollary \ref{large_distance_HSs_are_generic} is equivalent to that of Corollary \ref{large_distance_HSs_are_generic_intro}.

\vskip5pt

Before proving  Corollary \ref{large_distance_HSs_are_generic}, we will state a more general fact about generic sets in the subset $SNOW(\Sigma)$ of $\cal{PML}(\Sigma)^2$:

\vskip5pt

\begin{proposition}
\label{high-distance-subsets}
Let $H$ be a subset of $SNOW(\Sigma)$, and assume that $H_n = H \cap SNOW_n(\Sigma)$ has the following property:

\begin{enumerate}
\item[$(\ast)$] For every pair of transverse maximal train tracks $\tau, \tau'$ the set $H_n \cap \bf{G}^{n}\bf{P}(\tau) \times \bf{G^2P}(\tau')$ is dense in $\bf{G}^{n}\bf{P}(\tau) \times \bf{G^2P}(\tau')$.
\end{enumerate}

Then $H_n$ is generic in  $H$.

\end{proposition}

\begin{proof} 
This follows directly from the observation that the hypothesis $(\ast)$ on $H_n$ is the only ingredient necessary in order to be able to apply the proof of Theorem  \ref{large_distance_snows_are_generic} word-by-word, with $SNOW_n(\Sigma)$ replaced by $H_n$ and $SNOW(\Sigma)$ replaced by $H$.

\end{proof}

\begin{proof} [Proof of Corollary \ref{large_distance_HSs_are_generic}]
Apply Proposition \ref{high-distance-subsets} to $H = \sigma_{\cal {H}}(\cal H^g)$ and 
$H_n = \sigma_{\cal {H}}(\cal H_n^g)$: It suffices to show that property $(\ast)$ holds, in order to conclude that  $\sigma_{\cal {H}}(\cal H_n^g)$ is generic in $\sigma_{\cal {H}}(\cal H^g)$.  

We now prove statement  $(\ast)$: For this purpose we need to show that any non-empty open set 
$U \subset \bf{G}^{n}\bf{P}(\tau) \times \bf{G^2P}(\tau')$ contains a point from $\sigma(\cal H_n^g)$.

Consider any pair $(\cal L, \cal L')$ of minimal-maximal uniquely ergodic laminations in 
$\bf{G}^{n}\bf{P}(\tau) \times \bf{G^2P}(\tau')$. Recall (see paragraph \ref{laminations-basics} (c) of Section \ref{tt}) that the set of such pairs is dense in  $\cal{PML}(\Sigma)^2$, and hence dense in 
$\bf{G}^{n}\bf{P}(\tau) \times \bf{G^2P}(\tau')$, since the latter is an open subset of $\cal{PML}(\Sigma)$. So $U$ contains such a pair $(\cal L, \cal L')$.

The open neighborhood $U$ of $(\cal L, \cal L')$ contains  the open set ${\bf P}(\tau_k) \times {\bf P}(\tau'_m)$ for sufficiently large $k, m \geq 1$, where $\tau_k$ is the train track obtained from $\tau$ 
by deriving $k$ times  in the direction of $\cal L$, and $\tau'_m$ is the train track obtained from $\tau'$ by by deriving $m$ times  in the direction of $\cal L'$. Without loss of generality, we can assume that $k \geq n$ and $m \geq 2$. Let $\tau_{k+1}$ and $\tau'_{m+1}$ be train tracks once more derived from $\tau_k$ and $\tau'_m$ respectively.

By Proposition \ref{SNOW-is-dense} the set $SNOW(\Sigma)$ is dense in $\cal{PML}(\Sigma)$. Since 
${\bf P}(\tau_{k+1}) \times {\bf P}(\tau'_{m+1})$ is open, there is a pair 
$(\cal D, \cal E) \in {\bf P}(\tau_{k+1}) \times {\bf P}(\tau'_{m+1})$ which is $SNOW$. We now conclude, by
Corollary \ref{pre-distance}, that the pair $(\cal D', \cal E') \in \sigma_{\cal {H}}(\cal H^g)$, which determines the same pair of handlebodies $V, W$ as $(\cal D, \cal E)$, satisfies 
$(\cal D', \cal E') \in {\bf P}(\tau_k) \times {\bf P}(\tau'_m) \subset U$. From our assumptions $k \geq n$ and  $m \geq 2$ we conclude that $(\cal D', \cal E') \in \sigma_{\cal {H}}(\cal H^g_n)$. 

This shows that $\sigma_{\cal {H}}(H^g_n) \cap ({\bf G^n P}(\tau) \times {\bf G^2 P}(\tau'))$ is dense in 
${\bf G^n P}(\tau) \times {\bf G^2 P}(\tau')$, as desired.

\end{proof}

 \end {document}